\documentclass{ws-join}

\usepackage{amsmath,amssymb}
\usepackage[utf8x]{inputenc}
\usepackage[T1]{fontenc}
\usepackage{lmodern}
\usepackage{microtype}
\usepackage{bbding}
\usepackage{color}
\usepackage{diagbox}
\usepackage{tikz}
\usetikzlibrary{arrows}
\usetikzlibrary{bending}

\newcommand{\F}{\mathbb{F}}

\newcommand{\diam}{{\rm diam} }

\newcommand{\dist}{{\rm dist} }
\newcommand{\fq}{\mathbb{F}_q}

\newcommand{\fqs}{\mathbb{F}_q^{\ast}}

\def\@begintheorem#1#2#3{\par\addvspace{8pt plus3pt minus2pt}%
             \noindent{\csname#1headfont\endcsname#1\hskip6pt\ignorespaces#3 #2.}%
              \csname#1font\endcsname\hskip.5em\ignorespaces}
\def\@endtheorem{\par\addvspace{8pt plus3pt minus2pt}\@endparenv}

\newtheorem{conjecture}{Conjecture}[section]

\newcommand{\nb}[1]{\gcd(#1,q-1)}

\title{A note on the Isomorphism Problem for Monomial Digraphs}

\author{
    ALEKSANDR KODESS\\
    Department of Mathematics\\
    University of Rhode Island\\
    Kingston, RI 02881, USA\\
    \texttt{kodess@uri.edu}
}
\author{
    FELIX LAZEBNIK\\
    Department of Mathematical Sciences\\
    University of Delaware\\
    Newark, DE 19716, USA\\
    \texttt{fellaz@udel.edu}
}

\begin{document}
\maketitle
\begin{center}
{\it Dedicated to the memory of Mirka Miller (1949 -- 2016)}
\bigskip
\end{center}
\begin{abstract}
Let $p$ be a prime
$e$ be a positive integer, $q = p^e$,
and let
$\F_q$ denote the finite field of $q$ elements.
Let $m,n$, $1\le m,n\le q-1$,  be integers.
The monomial digraph $D= D(q;m,n)$ is defined as follows:
the vertex set of $D$
is $\F_q^2$,
 and  $((x_1,x_2),(y_1,y_2))$ is an arc in $D$ if
$
x_2 + y_2 = x_1^m y_1^n
$.
In this note we study the question of isomorphism of monomial digraphs $D(q;m_1,n_1)$ and $D(q;m_2,n_2)$. 
Several necessary conditions and several 
sufficient conditions for the isomorphism are 
found. 
We  conjecture that one simple sufficient condition is also a necessary one.
\end{abstract}

\section{Introduction}
For all  terms related to digraphs which are not defined below, see Bang-Jensen and Gutin\cite{Bang_Jensen_Gutin}.
In this paper,
by a {\it directed graph} (or simply {\it digraph)}
$D$ we mean a pair $(V,A)$, where
$V=V(D)$ is the set of vertices and $A=A(D)\subseteq V\times V$ is the set of arcs.
For an arc $(u,v)$, the first vertex $u$ is called its {\it tail} and the second
vertex $v$ is called its {\it head}; we also denote such an arc by $u\to v$.
If $(u,v)$ is an arc,  we call $v$  an {\it out-neighbor}  of $u$, and $u$ an {\it in-neighbor} of $v$.
The number of out-neighbors of  $u$ is called the {\it out-degree} of $u$, and the number of in-neighbors of $u$ --- the {\it in-degree} of $u$.
For an integer $k\ge 2$,  a {\it walk} $W$ {\it from} $x_1$ {\it to} $x_k$ in $D$ is an alternating sequence
$W = x_1 a_1 x_2 a_2 x_3\dots x_{k-1}a_{k-1}x_k$ of vertices $x_i\in V$ and arcs $a_j\in A$
such that the tail of $a_i$ is $x_i$ and the head of $a_i$ is $x_{i+1}$ for every
$i$, $1\le i\le k-1$.
Whenever the labels of the arcs of a walk are not important, we use the notation
$x_1\to x_2 \to \dotsb \to x_k$ for the walk, and say that we have an $x_1x_k$-walk.
In a digraph $D$, a vertex $y$ is {\it reachable} from a vertex $x$ if there exists  a walk from $x$ to $y$ in $D$. In
particular, a vertex is reachable from itself. A digraph $D$ is {\it strongly connected}
(or, just {\it strong}) if, for every pair $x,y$ of distinct vertices in $D$,
$y$ is reachable from $x$ and $x$ is reachable from $y$.
A {\it strong component} of a digraph $D$ is a maximal induced subdigraph of $D$ that is strong.
If $x$ and $y$ are vertices of a digraph $D$, then the
{\it distance from x to y} in $D$, denoted $\dist(x,y)$, is the minimum length of
an $xy$-walk, if $y$ is reachable from $x$, and otherwise $\dist(x,y) = \infty$.
The {\it distance from a set $X$ to a set $Y$} of vertices in $D$ is
\[
\dist(X,Y) = \max
\{
\dist(x,y) \colon x\in X,y\in Y
\}.
\]
The {\it diameter}  of $D$ is defined as $\dist(V,V)$,  and it is denoted by $\diam(D)$.

Let $p$ be a prime, $e$ a positive integer, and $q = p^e$. Let
$\fq$ denote the finite field of $q$ elements, and  $\fq^*=\fq\setminus\{0\}$.

Let $\fq^2$
denote the Cartesian product $\fq \times \fq$, and let
 $f\colon\fq^2\to\fq$ be an arbitrary function. We define a  digraph $D = D(q;f)$  as follows:
 $V(D)=\fq^{2}$, and
there is an arc from a vertex ${\bf x} = (x_1,x_2)$ to a vertex
${\bf y} = (y_1,y_{2})$ if and only if
\[
x_2 + y_2 = f(x_1,y_1).
\]

If $({\bf x},{\bf y})$ is an arc in $D$, then ${\bf y}$ is uniquely determined by  ${\bf x}$ and $y_1$, and ${\bf x}$ is uniquely determined by  ${\bf y}$ and $x_1$.
Hence,  each vertex of $D$ has both its in-degree and out-degree equal to $q$.

By Lagrange's interpolation, 
$f$ can be uniquely represented by
a bivariate polynomial of degree at most $q-1$ in each of the variables.  If  ${f}(x,y) =  x^m y^n$, where $m,n$ are integers, $1\le m,n\le q-1$,  we call  $D$ a {\it monomial} digraph, and denote it
by  $D(q;m,n)$. Digraph $D(3; 1,2)$ is depicted in Fig.\,1.  As for every $a\in \fq$,  $a^q=a$, we will assume in the notation $D(q;m,n)$ that $1\le m,n\le q-1$.   It is clear,  that ${\bf x}\to {\bf y}$ in $D(q;m,n)$ if and only if ${\bf y}\to {\bf x}$ in $D(q;n,m)$. Hence, one digraph is obtained from the other by reversing the direction of every arc. In general, these digraphs are not isomorphic, but if one of them is strong so is the other  and their diameters are equal. Also, if one of them contains a path or cycle, then the other contains a path or a  cycle  of the same length.

\begin{figure}
\begin{center}
\begin{tikzpicture}
\tikzset{vertex/.style = {shape=circle,draw,inner sep=2pt,minimum size=.5em, scale = 1.0},font=\sffamily\scriptsize\bfseries}
\tikzset{edge/.style = {->,> = triangle 45}}

\node[vertex] (a) at  (2,0) {$(1,0)$};
\node[vertex] (b) at  (5,0) {$(0,0)$};
\node[vertex] (c) at  (8,0) {$(2,0)$};

\node[vertex] (d) at  (0,-2) {$(2,1)$};
\node[vertex] (e) at  (2,-2) {$(1,1)$};
\node[vertex] (f) at  (8,-2) {$(2,2)$};
\node[vertex] (g) at  (10,-2) {$(1,2)$};

\node[vertex] (h) at  (2,-4) {$(0,2)$};
\node[vertex] (i) at  (8,-4) {$(0,1)$};

\draw[edge] (a) to (b);
\draw[edge] (a) to (d);
\draw[edge] (a) to (e);

\draw[edge] (b) to (a);
\draw[edge] (b) to (c);

\draw[edge] (c) to (b);
\draw[edge] (c) to (f);
\draw[edge] (c) to (g);

\draw[edge] (d) to (e);
\draw[edge] (d) to (h);

\draw[edge] (e) to (a);
\draw[edge] (e) to (c);
\draw[edge] (e) to (h);

\draw[edge] (f) to (a);
\draw[edge] (f) to (c);
\draw[edge] (f) to (i);

\draw[edge] (g) to (f);
\draw[edge] (g) to (i);

\draw[edge] (h) to (d);
\draw[edge] (h) to (e);
\draw[edge] (h) to (i);

\draw[edge] (i) to (f);
\draw[edge] (i) to (g);
\draw[edge] (i) to (h);

\path
    (b) edge [->,>={triangle 45[flex,sep=0pt]},loop,out=240,in=270,looseness=7.5] node {} (b);
\path
    (d) edge [->,>={triangle 45[flex,sep=0pt]},loop,out=160,in=130,looseness=8] node {} (d);
\path
    (g) edge [->,>={triangle 45[flex,sep=0pt]},loop,out=20,in=50,looseness=8] node {} (g);
\end{tikzpicture}
\caption{The digraph $D(3;1,2)$: $x_2+y_2 = x_1y_1^2$.}
\end{center}
\end{figure}
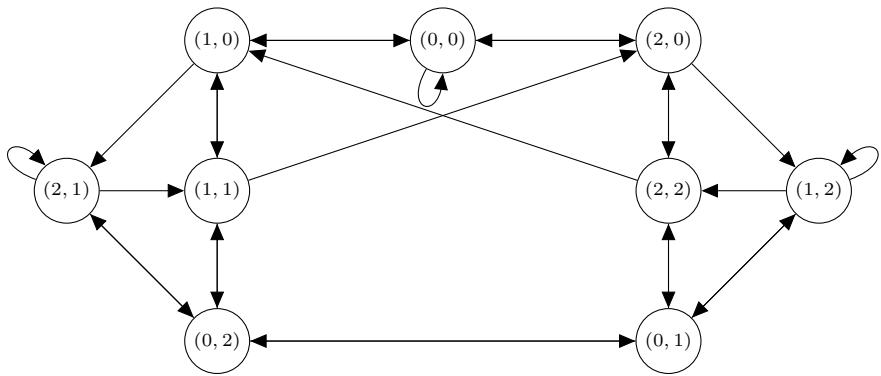

The digraphs $D(q; {f})$
and $D(q;m,n)$ are directed analogues
of 
some algebraically defined graphs,  which have been studied extensively
and have many applications: see
Lazebnik and Woldar\cite{LazWol01}, and  a recent survey by Lazebnik, Sun, and Wang\cite{LazSunWan16}.

The study of digraphs $D(q; f)$ started with the questions of connectivity and diameter. 
The questions of strong connectivity of digraphs $D(q;{f})$ and $D(q; m,n)$ and descriptions of their components were completely answered by Kodess and Lazebnik\cite{Kod_Laz_15}. The problem of determining the diameter of a component of $D(q;{f})$ for an arbitrary prime power $q$ and an arbitrary $f$ turned out to be
rather difficult. A number of results concerning  some instances of this problem for strong monomial digraphs were obtained by Kodess, Lazebnik, Smith, and Sporre\cite{KLSS16}.

As the order of $D(q; m,n)$ is $q^2$,  it is clear that digraphs $D_1=D(q_1; m_1,n_1)$  and $D_2= D(q_2; m_2,n_2)$ are isomorphic (denoted as $D_1\cong D_2$)  only if $q_1=q_2$.  Hence, the isomorphism problem for monomial digraphs can be stated as follows:  find  necessary and sufficient conditions on $q,m_1,n_1,m_2,n_2 $ such that $D_1\cong D_2$.   Though we are still unable to solve the problem,  all our partial results support the following conjecture.

\begin{conjecture}
\label{main_conj} {\rm [Kodess\cite{Kod14}]}
{\it 
Let $q$ be a prime power, and let $m_1,n_1,m_2,n_2$ be be integers from $\{1,2,\ldots, ,q-1\}$.
Then  $D(q; m_1,n_1)\cong D(q; m_2,n_2)$  if and only if there exists an integer $k$,
coprime with $q-1$, such that
\[
m_2 \equiv k m_1 \mod (q-1),
\]
\[
n_2 \equiv k n_1 \mod (q-1).
\]
}
\end{conjecture}
The sufficiency part of the conjecture is easy to demonstrate,  and we do it in  Section 3 (Theorem \ref{conj_suff}). We verified the necessity of these conditions with a computer for all prime powers $q$, $2\le q\le 97$. 
In the case $m_1 = m_2 = 1$ (hence,  $n_1 = n_2$),   the necessity of the conditions 
 was verified for all odd prime powers $q$, $3\le q\le 509$.

Our interest in the isomorphism problem for monomial digraphs
$D(q; m,n)$ and Conjecture \ref{main_conj} is two-fold. 
First, due to  the existence of  a simple isomorphism criterion   for  similarly constructed bipartite graphs $G(q; m,n)$ (see Theorem \ref{DLV} below) defined as follows.  Each partition of the vertex set of $G(q; m,n)$, which are denoted by $P$
and $L$, is a copy of $\fq^2$,    and two vertices
$(p_1, p_2)\in P$ and $(l_1, l_2)\in L$ are
adjacent if and only if $p_2 + l_2 = p_1^m l_1^n$. 
Secondly, due to applications of graphs $G(q; m,n)$ and their 
generalizations 
to a number of problems in extremal graph theory\cite{LazSunWan16}$^,$\cite{LazWol01}. 
We also note that our conjecture is similar in spirit to
the 1967 conjecture of \'{A}d\'{a}m's\cite{Adam} that states that two circulant 
graphs ${\rm Cay}(\mathbb{Z}_n,S)$ and ${\rm Cay}(\mathbb{Z}_n,T)$, 
$T,S\subseteq\mathbb{Z}_n$, are isomorphic if and only if $S = mT$ for some 
$m\in\mathbb{Z}^{\ast}_n$. While \'{A}d\'{a}m's conjecture  
was soon shown to be false (Elspas and Turner\cite{Elspas_Turner}), a number of questions surrounding it 
have since drawn considerable attention (see surveys of Klin, Muzychuk, and P\"{o}schel\cite{Klin_Muzy_Pos}, and P\'{a}lfy\cite{Palfy}, or 
more recent works of 
Muzychuk\cite{Muz_1999}$^,$\cite{Muz_2004}, and 
Evdokimov and Ponomarekno\cite{Evd_Pon}).
\smallskip

For any integer $a$, we let $\nb{a}$ denote the greatest common divisor of $a$ and $q-1$.

\begin{theorem}\label{DLV}{\rm [Dmytrenko, Lazebnik, and Viglione\cite{DLV05}]}
{\it 
$G(q; m_1,n_1)\cong G(q; m_2,n_2)$ if and only if
$\{ \nb{m_1}, \nb{n_1}\} = \{ \nb{m_2}, \nb{n_2}\}$
as multisets.
}
\end{theorem}

For every digraph $D(q; f)$,  one can define a bipartite graph $G(q;f)$,
the {\it bipartite cover} of $D(q; f)$, in the following way.
Each partition $X$ and $Y$  of the  vertex set of $G(q;f)$ is defined to be a copy of $V(D(q;f))$,  and a vertex ${\bf x}=(x_1,x_2)\in X$ is joined to a vertex  ${\bf y}=(y_1,y_2)\in Y$  in $G(q;f)$ if and only if ${\bf x}\to {\bf y}$ in $D(q;f)$.
This construction is of special interest to us in view of the following
proposition that  provides us with the first non-trivial necessary condition for isomorphism of monomial digraphs.
\begin{proposition} \label{prop1}
 If $D(q; m_1,n_1)\cong D(q; m_2,n_2)$, then
$G(q; m_1,n_1)\cong G(q; m_2,n_2)$  and $\{ \nb{m_1}, \nb{n_1}\} = \{ \nb{m_2}, \nb{n_2}\}$ as multisets.
\end{proposition}
The  first part of Proposition \ref{prop1} 
simply states that two isomorphic digraphs have isomorphic bipartite covers, 
and the second part follows from Theorem \ref{DLV}.
In contrast to the case of the monomial bipartite graphs,  this necessary condition of Proposition \ref{prop1} is far from being sufficient for the isomorphism of monomial digraphs.
\medskip

In the following section we discuss some general properties of  isomorphisms of monomial digraphs. In Section 3 we prove the sufficiency of Conjecture \ref{main_conj} and present several  other sufficient or  necessary conditions on the parameters of isomorphic monomial digraphs.  In Section 4 we finish the note with some concluding remarks.

\section{Some general properties of isomorphisms of  monomial digraphs}

Suppose digraphs $D_1 = D(q; m_1,n_1)$ and $D_2 = D(q; m_2,n_2)$ are isomorphic via an isomorphism $\phi\colon V(D_1) \to V(D_2)$, $(x,y)\mapsto \phi((x,y)) = (\phi_1((x,y)), \phi_2((x,y))$. For brevity, we will write $\phi(x,y)$ for $\phi((x,y))$. 
Functions $\phi_1$ and $\phi_2$ can be considered polynomial functions of two variables on $\fq$  of degree 
at most $q-1$ with respect to each variable. 

A polynomial $h\in \fq [X_1,\ldots, X_n]$  is called a {\it permutation polynomial} in $n$ variables on $\fq$  if the equation $h(x_1,\ldots,x_n)=\alpha$ has exactly $q^{n-1}$ solutions in $\fq^n$ for each $\alpha\in \fq$.  For $n=1$, this definition implies that the  function  on $\fq$ induced by $h$ is a bijection,  and in this case $h$ is called just a {\it permutation polynomial} on $\fq$.

The following
theorem describes some properties of the functions induced by the polynomials $\phi_1 = f$ and $\phi_2 = g$, and imposes a strong restriction  on  the form of $g$.
\begin{theorem}
\label{thm:f_and_g}
{\it 
Let $q$ be an odd prime power,  $D_1 = D(q; m_1,n_1)\cong D_2 = D(q; m_2,n_2)$ with an isomorphism
given in the form
\[
\phi\colon V(D_1)\to V(D_2),\;
(x,y)\mapsto (f(x,y),g(x,y))
\]
for some $f,g\in\fq[X,Y]$ of degree at most $q-1$ in each of the variables.
Then the following statements hold.
\begin{enumerate}
\item[(i)]
 $f$ and $g$ are permutation polynomials in two variables on $\fq$.
\item[(ii)]
If $m_1\neq n_1$, then $f(x,y) = 0$ if and only if $x = 0$.
\item[(iii)] 
If $m_1\neq n_1$, then $g$ is a polynomial of indeterminant $Y$ only, and is of the form
\[
g(Y) = a_{q-2}Y^{q-2} + a_{q-4}Y^{q-4} + \dots + a_1 Y,
\]
where all $a_i\in\fq$, $i=1,\dotso,q-2$. Moreover, $g$ is a permutation polynomial on $\fq$.
\end{enumerate}
}
\end{theorem}
\begin{proof}
   As $\phi$ is a bijection, the system
    \[
        \left\{
        \begin{array}{ccc}
            f(x,y) & = & a,\\
            g(x,y) & = & b,
        \end{array}
        \right.
    \]
    has a solution for every pair $(a,b)\in\fq^2$.
    Fix an $a$ and let $b$ vary through all of
    $\fq$. This gives $q$ distinct solutions
    $(x_i, y_i)$, $i=0,\dots, q-1$, of the system. Note that for every $i$,
    we have $f(x_i, y_i)=a$, so these are $q$ distinct
    points at which $f$ takes on the 
    value $a$.
    Assume that for some $(x^*,y^*)$ distinct  from each
    $(x_i,y_i)$ we have $f(x^*,y^*)=a$. As $g(x_i,y_i)$
    runs through all of $\fq$, we have
    $g(x^*,y^*) = g(x_i,y_i)$ for some $i$. Then, for this $i$, we have
    \[
        \phi (x^*,y^*)=
        \Bigr(f(x^*,y^*),g(x^*,y^*)\Bigr)=
        \Bigr(f(x_i,y_i),g(x_i,y_i)\Bigr)=
        \phi(x_i,y_i),
    \]
    contradicting the choice of $(x^*,y^*)$. Hence, the equation  $f(x,y)=\alpha$  has exactly $q$ solutions for each $\alpha\in \fq$,  and so $f$ is a permutation polynomial in two variables on $\fq$.  The proof of the statement for $g$ is similar. This proves part (i).

\bigskip

Since $\phi$ is an isomorphism, the following two equations
\begin{equation}
    \label{eqn:adjacency}
    x_2 + y_2 = x_1^{m_1} y_1^{n_1},
\end{equation}
\begin{equation}
    \label{eqn:isomor_adj}
    g(x_1, x_2) + g(y_1, y_2) =
    f(x_1, x_2)^{m_2} \cdot f(y_1, y_2)^{n_2}
\end{equation}
are equivalent.

From \eqref{eqn:adjacency}, $y_2 = x_1^{m_1} y_1^{n_1} - x_2$, and
substituting this expression for $y_2$ in \eqref{eqn:isomor_adj} we have
    \begin{equation}
        \label{eqn:mn_plugged}
        g(x_1, x_2) + g(y_1, x_1^{m_1} y_1^{n_1} - x_2) =
        f(x_1, x_2)^{m_2}\cdot
        f(y_1, x_1^{m_1} y_1^{n_1} - x_2)^{n_2},
    \end{equation}
for all
$x_1,x_2,y_1\in\fq$.  Let $(a,b)\in \fq^2$ be such that $f(a,b)=0$  (its existence follows from part (i)).  Set $(x_1, x_2) = (a,b)$,
and set  $y_1 = s$. Then \eqref{eqn:mn_plugged} yields
    \begin{equation}
        \label{eqn:amtn}
        g(a,b) + g(s, a^{m_1} s^{n_1} - b) = 0,\quad\text{for all }\, s\in\fq.
    \end{equation}
Likewise from \eqref{eqn:adjacency}, $x_2 = x_1^{m_1} y_1^{n_1}-y_2$.  Substituting this expression for $x_2$ in \eqref{eqn:isomor_adj}, and
setting $x_1 = t$ and  $(y_1, y_2) = (a,b)$,
we obtain
    \begin{equation}
        \label{eqn:tman}
       g(t, t^{m_1} a^{n_1} - b) + g(a,b)  = 0,\quad \text{for all }\, t\in\fq.
    \end{equation}
Hence, \eqref{eqn:amtn} and \eqref{eqn:tman} yield
    \[
        g(s, a^{m_1} s^{n_1} - b) =
        g(t, a^{n_1} t^{m_1} - b) =
        -g(a,b),\quad\text{for all }\, s, t\in\fq.
    \]
From part (i), $g$ is a permutation polynomial in two variables, and
we conclude that  the set
\[
\{(s, a^{m_1} s^{n_1} -b),(t, a^{n_1} t^{m_1} - b)\,:\, s,t\in \fq\}
\]
contains exactly $q$ elements. As $(s, a^{m_1} s^{n_1} -b)=(t, a^{n_1} t^{m_1} - b)$
implies $s=t$, we obtain that
$a^{m_1} t^{n_1} - b = a^{n_1} t^{m_1} - b$ for all $t\in\fq$. Since $m_1\neq n_1$,
this implies $a = 0$ as, otherwise the polynomial
$X^{|m_1-n_1|}-a^{|m_1-n_1|}\in\fq[X]$ of degree $|m_1-n_1|\le q-2$ has
$q$ roots.

Thus, $f(a,b) = 0$ implies $a = 0$. From part (i),  $f$ is a permutation polynomial in two variables. Let $\{ (a_i,b_i)\}_{i=1}^q$ be the
set of $q$ distinct points at which $f$ is zero. Then $a_i = 0$ for all
$i$, and all $b_i$ must be distinct. That is, $f(0,b) = 0$ for any $b\in\fq$.
This proves part (ii).

\bigskip

We now turn to the proof of part (iii). We just concluded that $f(a,b) = 0$ implies $a = 0$. Substituting $a=0$ in  (\ref{eqn:tman}) we obtain
\begin{equation}
\label{eqn:gtb}
g(t,-b) = -g(0,b),\quad\text{ for all } b,t\in\fq.
\end{equation}
Write $g(X,Y) = Yg_1(X,Y) + \hat{g}(X)$ for some $g_1\in\fq[X,Y]$, and
$\hat{g}\in\fq[X]$ of degree at most $q-1$.
Now from (\ref{eqn:gtb}), we have $g(x,0) = \hat{g}(x) = -g(0,0)$
    for every $x\in \fq$. Since the degree of $\hat{g}$ is at most $q-1$,
    it follows that $\hat{g}(X)$ is a constant
    polynomial. Also from (\ref{eqn:gtb}), $g(0,0) = \hat{g}(0) = -g(0,0)$,
    and, as $q$ is odd, $\hat{g}$ is the zero polynomial. Thus $g(X,Y) = Y g_1(X,Y)$ for some
    $g_1\in\fq[X,Y]$, where the degree of $g_1$ in $Y$ is at most $q-2$.

Using (\ref{eqn:gtb}) again, we find that
\begin{equation}
\label{eqn:g1tb}
g_1(t,-b) = g_1(0,b),\quad\text{ for all } b\in\fqs,\,t\in\fq.
\end{equation}
Write $g_1(X,Y) = X h_1(X,Y) + h_2(Y)$, where
$h_1\in\fq[X,Y]$, $h_2\in\fq[Y]$.
By (\ref{eqn:g1tb}), for all $t\in\fq$ and
    all $b\in\fqs$ we have
    \begin{equation}
    \label{eqn:gh1h2}
        g_1(t,-b)=t h_1(t,-b)+h_2(-b)=
        g_1(0,b)=h_2(b).
    \end{equation}
    For $t=0$, it implies   that
    $h_2(b) = h_2(-b)$ for all $b\in\fqs$, and since $q$ is odd,  and 
    the degree of $h_2$ is at most $q-2 $, we have
    $h_2(Y) = \sum_{i=0}^{(q-3)/2} \tilde{a}_{i} Y^{2i}$
    for some $\tilde{a}_{i}\in\fq$, $0\le i\le (q-3)/2$.
    From (\ref{eqn:gh1h2}), it now follows that for every $t\in\fq$ and every
    $b\in\fqs$, $t h_1(t,-b)=0$, and so $h_1(t,-b)=0$ for all $b,t\in\fqs$.
    Write $h_1(X,Y)$ as
    \[
        h_1 = h_1(X,Y) =
        c_{q-2}(Y)X^{q-2} + c_{q-3}(Y)X^{q-3} +
        \dotsb +
        c_1(Y)X + c_0(Y),
    \]
    where all $c_i\in\fq[Y]$ are of degree at most $q-2$.
   For any fixed $b\in \fqs$,  the polynomial $h_1(-b, Y)$ of degree at most $q-2$ has $q-1$ roots. Hence,
    $c_i(-b)=0$ for all $i$, $0\le i\le q-2$, and so all $c_i(Y)$ are  zero polynomials. Thus, $h_1(X,Y)$ is the zero polynomial. Therefore,
    \[
    g(X,Y) =
    Yg_1(X,Y) = Y(X h_1(X,Y) + h_2(Y)) = Y h_2(Y) =
    \sum_{i=0}^{(q-3)/2} \tilde{a}_{2i} Y^{2i+1}.
    \]
     Set $a_{i+1} = \tilde{a}_{2i}$
    for all $i$, $0\le i\le (q-3)/2$, so
    \begin{equation}
    \label{eqn:shape_g}
    g(Y) = a_{q-2}Y^{q-2} + a_{q-4}Y^{q-4} + \dots + a_1 Y.
    \end{equation}
    Every  permutation polynomial in two variables,  which is actually a polynomial of one variable,  has to be a permutation polynomial. By part (i), and by the last expression for $g$ as  $g(Y)$,  we obtain that $g$ is a permutation polynomial.   This ends the proof  of part (iii), and of the theorem.
\end{proof}
Theorem \ref{thm:f_and_g} will be used in the proof of Theorem~\ref{thm:bars} of the next section.

\section{Conditions on the parameters of isomorphic monomial digraphs}

We begin with the proof of the sufficiency part of Conjecture \ref{main_conj}, and provide several more sufficient conditions for 
the isomorphism of monomial digraphs in Corollary \ref{cor1}. 
Then we obtain 
several necessary conditions for monomial digraphs to be isomorphic. 
\begin{theorem}
\label{conj_suff}
{\it 
Suppose  there exists an integer $k$
 such that $\nb{k} = 1$ and
\[
m_2 \equiv k m_1 \mod (q-1),
\]
\[
n_2 \equiv k n_1 \mod (q-1).
\]
Then $D(q; m_1,n_1)\cong D(q; m_2,n_2)$.
}
\end{theorem}
\begin{proof}
Define the mapping $\phi\colon V(D(q;m_2,n_2))\to V(D(q;m_1,n_1))$ via the
rule
\[
\phi \colon (x,y)\mapsto (x^k,y).
\]
As $\nb k =1$, $\phi$ is bijective and we check that $\phi$ preserves adjacency and non-adjacency. Let $(x_1,x_2)\to (y_1,y_2)$  in $D(q;m_2,n_2)$. Then
 $x_2 + y_2 = x_1^{m_2} y_1^{n_2}$. We have
\[
\phi(x_1,x_2)  = (x_1^k,x_2),
\]
\[
\phi(y_1,y_2)  = (y_1^k,y_2),
\]
and
\[
x_2 + y_2
=
x_1^{m_2} y_1^{n_2}\;\;
\Leftrightarrow \;\;
x_2 + y_2
= (x_1^k)^{m_1} (y_1^k)^{n_1}.
\]
Hence, $(\phi(x_1,x_2),\phi(y_1,y_2)) = ((x_1^k,x_2),(y_1^k,y_2))$ is an arc in $D(q;m_1,n_1)$,
and $\phi$ is indeed an isomorphism from $D(q;m_2,n_2)$ to $D(q;m_1,n_1)$.
\end{proof}

\begin{corollary}
\label{cor1}
The following statements hold.
\begin{enumerate}
\item[(i)]
If $\nb{m} = 1$,
then
$D(q; m,n) \cong D(q; 1,n')$,  for some integer $n'$ such that $mn'\equiv n \mod (q-1)$.
\item[(ii)]
If $m n \equiv 1 \mod (q-1)$, then
$D(q; m,1) \cong D(q; 1,n)$, and  $D(q; m,n) \cong D(q; 1,n^2)\cong D(q; m^2, 1)$.
\item[(iii)] If $m+n\equiv 0 \mod (q-1)$, then $D(q; m,n)\cong D(q;n,m)$.
\item [(iv)] If $D(q;m_1,n_1) \cong D(q;m_2,n_2)$ and $m_1=n_1$, then $m_2=n_2$, and $\nb{m_1}=  \nb{m_2}$.
\item[(v)] If  $\nb{m}=  \nb{n}$, then  $D(q; m,m) \cong D(q;n,n)$.
\end{enumerate}
\end{corollary}
\begin{proof}
Part (i) is straightforward. As $\nb{m} = 1$,  there exists an integer $k$ such that $\nb k = 1$ and $1\equiv km \mod (q-1)$. Let  $n' \equiv kn \mod (q-1)$.  By Theorem \ref{conj_suff},  $D(q; m,n) \cong D(q; 1,n')$.

For part (ii),
 $m n \equiv 1 \mod (q-1)$ is equivalent to   $\nb{m} = \nb{n} = 1$,  and the conclusion  follows directly
from Theorem \ref{conj_suff} by taking $k$ equal $m$ or $n$.

For part (iii),  we need to show that $D(q; m,-m)\cong D(q;-m,m)$. As $\nb{-1} =1$, the statement follows from
Theorem \ref{conj_suff}.

Let us prove  part (iv).  If $m_1=n_1$,  then for every arc of $D_1$,  the opposite arc is also an arc of $D_1$. As $D_1\cong D_2$, for every arc of $D_2$,  the opposite arc is also an arc of $D_2$. Consider an arc
of $D_2$ of the form $(a,b)\to (1,a^{m_2} - b)$. Then $D_2$ contains the opposite arc $(1,a^{m_2} - b)\to (a,b)$ only if $a^{n_2}=a^{m_2}$.  Taking $a$ to be a primitive element of $\fq$, we obtain $m_2=n_2$.  Then the equality  $\nb{m_1}=  \nb{m_2}$ follows from Proposition \ref{prop1}.

For part (v), let $m_1=n_1=m$ and $m_2=n_2=n$. 
We use the following number-theoretic result: 
if $\nb{m} = \nb{n}$, then there exists an integer $k$ 
coprime with $q-1$ such that $mk \equiv n \mod (q-1)$. 
For a proof of a more general related result see\cite{DLV05} or Viglione\cite{Viglione_thesis}.  
Hence the conditions of Theorem \ref{conj_suff} are met, 
and $D(q; m,m) \cong D(q;n,n)$.
\end{proof}

The following statement provides some information on the automorphism groups of monomial digraphs.
The proof is trivial,  and we omit it.

\begin{proposition}\label{proppsi}
For any $c\in\F_q^{\ast}$, the mapping
$\psi_c\colon (x,y)\mapsto (cx,c^{m+n} y)$ is
an automorphism of $D(q; m,n)$. In particular,
the group of automorphisms of $D(q; m,n)$
contains a cyclic subgroup of order $q-1$ generated
by $\psi_g$, where $\langle g \rangle = \F_q^{\ast}$.
\end{proposition}

It is well known that  $\fqs$, viewed as a multiplicative group,   is a cyclic group of order $q-1$.  For any integer $n$, let
\[
A_n = \{x^n\colon x\in\fqs \},\quad
I_n = \{x\in\fqs\colon x^n = 1\}.
\]
By standard theory of cyclic groups, $|A_n| = (q-1)/\nb{n}$, and $|I_n| = \nb{n}$.
\medskip

In the following theorem we collect some independent  necessary conditions on the parameters of isomorphic monomial digraphs.

\begin{theorem}
\label{thm:bars}
{\it 
Let $D_1 = D(q; m_1,n_1)$, $D_2 = D(q; m_2,n_2)$ and $D_1\cong D_2$, where
$q$ is an odd prime power.
Then
\begin{enumerate}
\item[(i)]  $\nb{m_1} = \nb{m_2}$ and $\nb{n_1} = \nb{n_2}$.
\item[(ii)]
$
\nb{m_1+n_1}
=
\nb{m_2+n_2}
.
$
\item[(iii)]
$
\displaystyle
\nb{m_1-n_1}
=
\nb{m_2-n_2}
.
$
\end{enumerate}

Moreover,  the conditions (i) -- (iii) are independent in the sense that no two  of them imply the remaining one.
}
\end{theorem}
\begin{proof}
For (i), by Proposition \ref{prop1},  we have   $\{\nb{m_1}, \nb{n_1}\} = \{\nb{m_2}, \nb{n_2}\}$ as multisets.  Therefore,  in order to prove both equalities in (i), it is sufficient to prove only one of them.  We will show that $\nb{n_1} = \nb{n_2}$.

Let $\phi : D_1\to D_2$ be an isomorphism.  It follows from  Theorem \ref{thm:f_and_g}  that $\phi(0,0) = (0,0)$.  As $(1,0)$ is an out-neighbor of $(0,0)$ in $D_1$,  $\phi(1,0)$ is an out-neighbor of $(0,0)$ in $D_2$, distinct from $(0,0)$. The adjacency equation  in $D_2$  implies that  $\phi(1,0) = (c,0)$, for some $c\in \fqs$. By Proposition \ref{proppsi}, composing $\phi$ with $\psi_{c^{-1}}$, we obtain an isomorphism  $\phi_1: D_1\to D_2$, such that   $(0,0)\mapsto (0,0)$ and $(1,0)\mapsto (1,0)$.  Let $f$ and $g$ be the polynomials described in Theorem \ref{thm:f_and_g}, so that $\phi_1 ((a,b)) = (f(a,b), g(b))$ for every $(a,b)\in V(D_1)$.

The out-neighbors of the vertex $(1,0)$ distinct from $(0,0)$ in $D_1$ and in $D_2$ have
the form $(x,x^{n_1})$ and $(x,x^{n_2})$, respectively, for every  $x\in\fqs$.
As $\phi_1$ maps $(0,0)$ to $(0,0)$ and $(1,0)$ to $(1,0)$,
we obtain that for every $x\in\fqs$ there exists a unique $y\in\fqs$ such that
$(f(x,x^{n_1}),g(x^{n_1})) = (y,y^{n_2})$.
As $g$ is a permutation polynomial on $\fq$, and $g(0)=0$, we obtain that $g(A_{n_1}) = A_{n_2}$,  and so $|A_{n_1}| = |A_{n_2}|$.  As $|A_{n_i}|= (q-1)/{\nb n_i}$, $i=1,2$, we obtain $\nb{n_1}=\nb{n_2}$. This ends the proof of (i).

\bigskip

 For (ii), we count the number of distinct nonzero second coordinates of the vertices of $D = D(q;m,n)$ which  have a loop on them.   As $q$ is odd, there exists a loop on  a vertex $(x,y)$ of $D$ if and only if
\[
(x,y)\to (x,y)
\Leftrightarrow
2y = x^{m+n}
\Leftrightarrow
y = \frac{1}{2}x^{m+n}
\Leftrightarrow
(x,y) = (x,\frac{1}{2}x^{m+n}).
\]
Therefore, the number of distinct nonzero second coordinates of the vertices of $D$ which  have a loop on them  is
\[ 
|A_{m+n}| = \frac{q-1}{\nb{m+n}}.
\]

Now, if $\phi\colon D_1\to D_2$ is an isomorphism, then $\phi$ maps the set of loops of
$D_1$ to the set of loops of $D_2$ bijectively. As $\phi(0, 0) =(0, 0)$,  and both $D_1$ and $D_2$ have  a loop on $(0,0)$,  an argument similar to that of part (i) 
(based on the fact that $g$ is a permutation polynomial and $g(0) = 0$) yields $|A_{m_1+n_1}|= |A_{m_2+n_2}|$. Hence,
$\nb{m_1+n_1} = \nb{m_2+n_2}$, and part (ii) is now proved.

\bigskip

For (iii), we compute  the number of 2-cycles in $D = D(q;m,n)$, which we  denote by $c_2 = c_2(q;m,n)$. 
If $(x_1,x_2)\to(y_1,y_2)\to(x_1,x_2)$ is a 2-cycle in $D$, then
\begin{equation}
\label{eqn:2cycle}
x_2+y_2 = x_1^m y_1^n = x_1^n y_1^m,\quad
(x_1,x_2) \neq (y_1,y_2).
\end{equation}
To compute $c_2$, we count the number of solutions
$(x_1,x_2,y_1,y_2)\in\fq^4$ of this system.

There are $q(q-1)$ solutions if
$x_1 = 0$ and $y_1 \neq 0$, and the same number if $x_1 \neq 0$ and
$y_1 = 0$. If $x_1 = y_1 = 0$, then $x_2 = -y_2$ can be chosen in
$q-1$ ways. Thus there are
\[
q(q-1) + q(q-1) + (q-1) = (2q+1)(q-1)
\]
solutions with $x_1y_1 = 0$.

If $x_1 y_1 \neq 0$, then (\ref{eqn:2cycle}) implies
$(x_1 y_1^{-1})^{m-n} = 1$. If $x_1 = y_1$, then choose
$x_2 \neq \frac{1}{2}x_1^{m+n}$ in $q-1$ ways, so the value of
$y_2$ is determined uniquely and is different from $x_2$. This case
yields $(q-1)^2$ solutions. If $x_1 \neq y_1$, $x_1$ can be chosen in
$q-1$ ways, and $y_1$ in $|I_{m-n}|-1 = \nb{m-n}-1$ ways, and $x_2$ in
$q$ ways. Hence, in total there are
\[
(2q+1)(q-1) + (q-1)^2 + q(q-1)(\nb{m-n}-1) = q(q-1)(2+\nb{m-n})
\]
solutions to (\ref{eqn:2cycle}). As vertices $(x_1,x_2)$ and $(y_1,y_2)$
can we swapped in this count, the number of 2-cycles is half of this:
$$
c_2(q;m,n) = \frac{1}{2}q(q-1)(2+\nb{m-n}).
$$

If $D_1$ and $D_2$ are isomorphic, they have the same number of
2-cycles, and $c_2(q;m_1,n_1) = c_2(q;m_2,n_2)$ yields
$\nb{m_1-n_1} = \nb{m_2-n_2}$, ending the proof of part~(iii).

\bigskip

We now show that conditions (i), (ii), and (iii) are independent.
Let $q = 11$. Then
$(m_1,n_1) = (1,1)$ and $(m_2,n_2) = (1,3)$ satisfy (i) and (ii), but
not (iii);
$(m_1,n_1) = (1,2)$ and $(m_2,n_2) = (1,4)$ satisfy (i) and (iii), but
not (ii);
$(m_1,n_1) = (1,2)$ and $(m_2,n_2) = (1,10)$ satisfy (ii) and (iii), but
not (i).
\end{proof}

\begin{remark}
The conditions of Theorem \ref{thm:bars} do not imply those of
Conjecture \ref{main_conj}.
For instance,
let $m_1 = m_2 = 1$, $n_1 = 4$, $n_2 = 12$ with $q = 17$. 
Then 
$\nb{m_1} = \nb{m_2} = 1$,
$\nb{n_1} = \nb{n_2} = 4$,
$\nb{m_1+n_1} = \nb{m_2+n_2} = 1$, and
$\nb{m_1-n_1} = \nb{m_2-n_2} = 1$.
The digraphs $D(17; 1,4)$ and $D(17; 1,12)$ are not isomorphic, for 
otherwise they have the same number of isomorphic copies of the digraph shown in 
Fig.~\ref{fig:looped_path}. This in turn implies, via a discussion in 
Coulter, De Winter,  Kodess, and  Lazebnik\cite{CDKL16}, 
that the trinomials 
$X^5 - 2X + 1$ and $X^{13} - 2X + 1$ have the same number of roots in $\F_{17}$. 
This however is easily seen to be false. Of course, the non-isomorphism 
of these digraphs can be also easily established by a computer. 
\end{remark}

\section{Concluding remarks}
Let $N(D,H)$ denote the number of isomorphic copies of digraph $H$ in digraph
$D$.
One can attempt to solve the isomorphism problem by finding a ``test digraph''
$H$ such that $N(D_1,H) = N(D_2,H)$ if and only if $D_1\cong D_2$.
Similarly, one can try to resolve the problem by finding a ``test family'' of
digraphs ${\mathcal H}$ satisfying $N(D_1,H) = N(D_2,H)$ for all $H\in {\mathcal H}$ if only
if $D_1\cong D_2$.
This approach was successful in the case of the aforementioned undirected
class of
graphs $G(q; m,n)$\cite{DLV05}$^,$\cite{Viglione_thesis}.  It is worth noting   that $K_{2,2}$ (same as 4-cycle) was a good ``test graph'' in that case: for fixed $m,n$ and sufficiently large $q$, the equality of numbers of 4-cycles in $G(q; m_1,n_1)$ and $G(q; m_2,n_2)$ implied the isomorphism of the graphs. In order to obtain the result for all $q$,  the number of copies of other $K_{s,t}$-subgraphs  had to be counted.
This approach however fails for monomial digraphs $D(q;m,n)$
when the ``test digraphs'' are strong directed cycles:  for every odd
prime power $q$, the digraphs
$D_1 = D(q; \frac{q-1}{2},q-1)$ and $D_2 = D(q; q-1,\frac{q-1}{2})$ are not isomorphic
by Theorem \ref{thm:bars}, but have the same number of strong directed cycles of any lengths, since
every arc ${\bf x}\to{\bf y}$ in $D_1$ corresponds to the arc ${\bf y}\to{\bf x}$ in $D_2$.
It can also be shown that conditions of Theorem \ref{thm:bars} imply
that $D(q;m_1,n_1)$ and $D(q;m_2,n_2)$ have equal number of copies isomorphic
to $\overrightarrow{K}_{2,2}$ with all arcs directed from one partition to the other, and so this digraph cannot be a ``test digraph'' either.

So far we were unable to find a good ``test family'' to replicate the success with  monomial bipartite graphs  for monomial digraphs. One difficulty is that counting $N(D,H)$  in monomial digraphs is much harder, even for small digraphs $H$.  Another difficulty was with finding good candidates for $H$, even after utilizing all necessary conditions and extensive  experiments with computer.

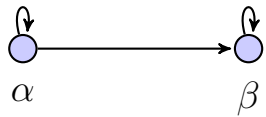
\begin{figure}
\begin{center}
\begin{tikzpicture}[->,>=stealth',shorten >=1pt,auto,node distance=3cm,thick,
every node/.style={circle,fill=blue!20,draw,font=\sffamily\Large\bfseries}]

  \node[label=below:$\alpha$] (1) {};
  \node[label=below:$\beta$] (2) [right of=1] {};

  \path[every node/.style={font=\sffamily\small}]
    (1) edge node [left] {} (2)
        edge [loop above] node {} (1)
    (2) edge [loop above] node {} ();
\end{tikzpicture}
\caption{The digraph $K$.}
\label{fig:looped_path}
\end{center}
\end{figure}

On the other hand, understanding the equality of  $N(D_1,K)= N(D_2,K)$ in monomial digraphs $D_1$ and $D_2$ for digraph $K$ of Fig.~2 led to a ``digraph-theoretic proof''  that  the numbers of solutions of certain polynomial equations over finite fields were equal,  and the latter was not clear to us at first from just algebraic considerations (see\cite{CDKL16}).
\medskip

\section*{Acknowledgement}    
The authors are thankful to the anonymous referees
whose thoughtful comments improved the paper. 
This work was partially  supported by a grant from the Simons Foundation ($\#$426092, Felix Lazebnik).

\end{document}